\documentclass[11pt]{amsart}
\usepackage{amssymb,amsmath,amsthm}
\usepackage{booktabs}
\newtheorem{theorem}{Theorem}

\newtheorem{conjecture}{\bf Conjecture}

\theoremstyle{remark}

\newtheorem*{remark}{Remark}

\numberwithin{theorem}{section} \numberwithin{equation}{section}

\newcommand{\Q}{\mathbb{Q}}

\newcommand{\Z}{\mathbb{Z}}

\begin{document}
\title[Small values of Lerch sums]{Small Values of Coefficients of a Half Lerch Sum }

\author{Xinhua Xiong} 
\address{Research Institute for Symbolic Computation\\Johannes Kepler University
\\ Altenberger Straße 69\\ A-4040 Linz,  Austria}
\email{xxiong@risc.uni-linz.ac.at}

\thanks{ The  author was  supported  by the Austria Science Foundation (FWF) grant SFB F50-06 (Special Research Programm `` Algorithmic and Enumerative Combinatorics'').}
\maketitle

\begin{abstract}
Andrews, Dyson and Hickerson proved many interesting properties of coefficients for a Ramanujan's $q$-hypergeometric series by relating it to real quadratic field $\Q(\sqrt{6})$ and using the arithmetic of $\Q(\sqrt{6})$,  hence solved a conjecture of Andrews on the distributions of its Fourier coefficients. Motivated by Andrews's conjecture, we discuss an interesting $q$-hypergeometric series which comes from a Lerch sum and rank and crank moments for partitions and overpartitions. We give Andrews-like conjectures for its coefficients. We obtain partial results on the distributions of small values of its coefficients toward these conjectures.
\end{abstract}

\section{introduction and statement of results}
The function
\begin{equation} \label{sigma}
\begin{aligned}
\sigma(q)&=\sum_{n=0}^{\infty} S(n)q^{n}=1+ \sum_{n=1}^{\infty}\frac{q^{\frac{n(n+1)}{2}} }{(-q)_n} \\
&=1+q-q^2+2q^3+ \dots +4q^{45}+ \dots+6q^{1609}+\dots
\end{aligned}
\end{equation}
first appeared in Ramanujan's lost note book \cite{Ra}.
It was noted that $S(n)$ can be interpreted  as the number of partitions of $n$ into distinct parts with even rank minus the number with odd rank. Recall that Dyson's rank of a partition is defined as its largest part minus its number of parts. Although it has a partition interpretation, it behaves completely different from most of partition functions.  Andrews \cite{andrewsMonthly} noted that most $q$-series with partition interpretations either have coefficients which tend to infinity in absolute value or are bounded. For instance, the usual partition function $p(n)$, whose generating function can be represented as a hypergeometric series, grows very fast. 

In his paper \cite{andrewsMonthly}, Georg Andrews studied $S(n)$ and conjectured that
\begin{enumerate} 
\item $S(n)$ is zero infinitely often,
\item  $\limsup|S(n)|=+\infty$.
\end{enumerate}
Andrews's conjecture was soon proved by himself, Dyson and Hickerson \cite{ADH}. They related $\sigma(q)$ to the arithmetic of quadratic field $\Q(\sqrt{6})$. They obtained an exact formula for $S(n)$, which implies that $S(n)$ is lacunary, i.e. its coefficients are almost always zero, and attains every positive integers infinitely many times. 

Besides the strange behaviors of its coefficients, the function $\sigma(q)$ is found related to automorphic forms. Cohen \cite{Co} used it to construct a classical Maass wave forms.  It has a representation as Hecke-type double sums
$$
\sigma(q)= \sum_{\substack{ n \geq 0 \\ |j| \leq n} } (-1)^{n+j}\, q^{\frac{n(3n+1)}{2} -j^2}
\left(1-q^{2n+1} \right).
$$
Zwegers \cite{zwegers} used Hecke-type double sums to describe mock theta functions, so Bringmann and Kane \cite{Br-Ka1} viewed it  as a ``false mock theta function". Along this line, more examples resembling $\sigma(q)$ are found and relations to quadratic fields  and automorphic forms are built, see \cite{BK2}, \cite{BK3}, \cite{CFLZ}, \cite{L}, \cite{Lo2} and  \cite{Lo1} etc..

In this paper, we will not go to the directions to relate more $q$-hypergeometric series to quadratic fields or automorphic forms. We will turn to the properties (1) and (2) considered by Andrews but for another strikingly simple-looking $q$-hypergeometric series which is defined by 
\begin{equation*} \label{hqdefn}
 h(q) :=\sum_{n=1}^{\infty} h(n)= \sum_{n=1}^{\infty} \frac{(-1)^{n+1} q^{n(n+1)/2}}{1-q^n}. 
 \end{equation*}
 $h(q)$ looks like $\sigma(q)$, and the distributions of its coefficients $h(n)$ appears to be similar to that of $S(n)$ in some faces.  But in other faces, it completely does not.  We will prove that $h(q)$ has property (1), but it is not lacunary. We conjecture it has property (2).
We note that  $h(q)$ is related to mock theta functions. Because the completed sum 
\begin{equation*} \label{hqdefn}
 \sum_{n=-\infty}^{-1} \frac{(-1)^{n+1} q^{n(n+1)/2}}{1-q^n} + \sum_{n=1}^{\infty} \frac{(-1)^{n+1} q^{n(n+1)/2}}{1-q^n}
  \end{equation*}
is almost the specialization of Lerch sum 
\begin{equation*}
\sum_{n\in\Z} \frac{(-1)^n e^{\pi i(n^2+n)\tau+2\pi inv}}{1-e^{2\pi in\tau + 2\pi iu}} 
\end{equation*} by taking $u=v=0$ and letting $ q=e^{2\pi i \tau}. $
This is why we call $h(q)$ a half Lerch sum. Lerch sums are building blocks for mock theta functions by the works of Zwegers \cite{zwegers}.  Secondly, $h(q)$ has a similar Hecke-type double sums form
$$
h(q)= \sum_{\substack{ n \geq 1 \\  1\leq j \leq n-1} } q^{n^{2} +nj} + \sum_{\substack{ n \geq 1 \\  0\leq j \leq n} } q^{n^{2} +nj}.
$$
But this double sum is different from usual Hecke double sums, here $j$ appears linearly in the power of $q$. 
Moreover, $h(q)$ is related to the first positive crank moments for partitions and overprtitions.  By \cite{ACK}, the generating function of the first positive crank moments for partitions can be represented as 
\begin{equation*} 
 \frac{1}{(q)_{\infty}} \sum_{n=1}^{\infty} \frac{(-1)^{n+1} q^{n(n+1)/2}}{1-q^n}.
 \end{equation*}
By \cite{LS}, the generating function of the first positive crank moments for overpartitions can be represented as 
\begin{equation*} 
 \frac{(-q)_{\infty}}{(q)_{\infty}} \sum_{n=1}^{\infty} \frac{(-1)^{n+1} q^{n(n+1)/2}}{1-q^n}.
 \end{equation*}
 We believe that $h(q)$ should be an important object in the study of $q$-hypergeometric series, automorphic forms and partitions theory.

Now we state our main results, we think they are strange and interesting.

\begin{theorem}\label{theorem1}
\begin{enumerate}
\item
For  $X>0$,
\begin{align}\label{inequality1}
\frac{X}{2} - \left(\frac{3}{2}+\sqrt{2}\right )\sqrt{X} -1-\frac{3\sqrt{2}}{2} <   \# \{ 1\leq n \leq X| h(n) =0   \} < 
\frac{3X}{4}-\frac{\sqrt{X}}{2\sqrt{2}}+1.
\end{align}
\item
For $n=p^{d}$, we have $h(n)=0$, where $p$ is an odd prime and $d$ is an odd positive integer.
\item
For $n=p_{1}p_{2}\dots p_{k}$, we have $h(n)=0$, where $p_{i}$ are different odd primes,  and $p_{k}> 2p_{1}p_{2}\dots p_{k-1}$, $k\geq 2$.
\end{enumerate}
\end{theorem}
\begin{remark}
Unlike $\sigma(q)$, which is lacunary, $h(q)$ is not lacunary.
\end{remark}


\begin{theorem}\label{theorem2}
\begin{enumerate}

\item
For sufficient large $X$,
\begin{align*}
  \frac{\sqrt{X}}{\log X}   \ll    \# \{ 1\leq n \leq X| h(n) =1   \}<
\left(1+\frac{1}{\sqrt{2}}\right)\sqrt{X}.
\end{align*}
\item
for $n=2^{k}$, we have $h(n)=1$, where $k \geq 1$ is a positive integer. 
\item
For $n=p^{2k}$, we have $h(n)=1$, where $p$ is an odd prime and is a positive integer.
\item
For $n=p_{1}^{2}p_{2}^{2}$ with $p_{1}$ and $p_{2}$ odd primes and $p_{2} > \sqrt{p_{1}}$, we have $h(n)=1$.\end{enumerate}
\end{theorem}

\begin{theorem}\label{theorem3}
\begin{enumerate}

\item
For sufficient large $X$,
\begin{align}\label{1-theorem3}
\frac{\sqrt{X}}{\log X} \ll \# \{ 1\leq n \leq X| h(n) =2   \} <
\frac{X}{2}+\left(\frac{3}{2}+\sqrt{2}\right)\sqrt{X} +1+\frac{3\sqrt{2}}{2}.
\end{align}
\item
For $n=p_{1}p_{2}$ with $p_{1}$ and $p_{2}$ odd prime and $p_{1} < p_{2} < 2p_{1} $, we have $h(n)=2$.
\item
For $n=p\times (p+1)$ or $(p-1)\times p$, with $p$ an odd prime, we have $h(n)=2$.
\end{enumerate}
\end{theorem}


\begin{theorem}\label{theorem4}
For $n=p_{1}^{2}p_{2}^{2}$ with $p_{1}$ and $p_{2}$ odd prime and $p_{1} < p_{2} < \sqrt{2}p_{1}$, we have $h(n)=3.$\end{theorem}
\begin{remark}
Our theorems above show that $h(n)$ can attain $0$, $1$, $2$, $3$ infinite many times.
Numerical experiments suggest more. See the conjectures in Section 3. 
\end{remark}

\begin{theorem}\label{theorem5}
$h(n)$ is odd if and only if $n$ is a square or a double square, hence for
 almost all positive integers $n$, $h(n) \equiv 0 \pmod 2$. 
\end{theorem}
\begin{remark}
 The partition function $p(n)$, its values distributions are simple, just growing fast,  but has strange parity (conjecturally equi-distributed see \cite{Ahl}, \cite{BYZ}, \cite{Ono}). The distributions of the values of $h(n)$ are strange,  but $h(n)$ has simple parity.
\end{remark}


The paper is organized as follows.  
In Section 2, we prove  our main theorems. 
In Section 3, we give further conjectures about the distributions of the values of $h(n)$.


\section{Proof of the main theorems}
We start with representing  $h(q)$ as
\begin{equation} \label{rep of h(n)}
\begin{aligned}
h(q)=\sum_{n=0}^{\infty} h(n)q^{n} 
=\sum_{n=1}^{\infty}\sum_{j=1}^{n-1} 2q^{n(n+j)}+\sum_{n=1}^{\infty} (q^{n^{2}}+q^{2n^{2}}).
\end{aligned}
\end{equation}
This is the Lemma 2.2 in \cite{ACKO}.
Define three sets $A_{n}$, $B_{n}$ and $C_{n}$ as follows:
\begin{equation*} 
\begin{aligned}
A_{n}=\left\{ (x,y)| n=x(x+y),    \,\mbox{ with } x, y \mbox{ are positive integers and  }\,  1\leq y\leq x-1.   \right\}
\end{aligned}
\end{equation*}
\begin{equation*} 
\begin{aligned}
B_{n}=\left\{ x \,| n=x^{2},    \,\mbox{ with } x  \mbox{ is an positive integer }.   \right\}
\end{aligned}
\end{equation*}
\begin{equation*} 
\begin{aligned}
C_{n}=\left\{ x \,| n=2x^{2},    \,\mbox{ with } x  \mbox{ is an positive integer }.   \right\}
\end{aligned}
\end{equation*}
By  (\ref{rep of h(n)}), 
\begin{equation}\label{h(n)}
h(n)=2\# A_{n} + \# B_{n} + \# C_{n}.
\end{equation}
\begin{proof}[Proof of Theorem \ref{theorem1}]\noindent (1)
We note that $h(n)=0$ if and only if $n$ is neither a square nor a double square or does not appear in the 
following sequeences of numbers for any $m\geq 1$,
$$
m^{2}+m, \, m^{2}+2m,\, m^{2}+3m,\,\dots , m^{2}+(m-1)x.
$$
So  $h(n)\neq 0$ if and only if $n$ appears in the following numbers for some $m$,
\begin{equation*}
\begin{split}
&1,\,\,\,\,\,\,\,\,2,\\
&2^{2},\, \,\,\,\,\,2^{2}+2,\,\,\,\,\, 2\times2^{2},\\
&3^{2},\, \,\,\,\,\,3^{2} +3,\, \,\,\,\,\,3^{2} +2\times3, \,\,\,\, 2\times 3^{2},\\
&\dots\,\,\,\,\,\,\,\,\,\dots\,\,\,\,\,\,\,\,\,\,\,\,\,\,\,\dots \,\,\,\,\,\,\,\,\,\,\,\,\,\,\,\,\,\,\,\,\,\,\,\,\,\, \dots \\
&m^{2}, \,\,\,\,\,\, m^{2}+m,\,\,\,\,\,\,\,m^{2}+2m,\,\,\,\,\,\,m^{2}+3m,\,\,\,\,\,\, \dots, \,\,\,\,\,\,2m^{2}.\\
\end{split}
\end{equation*}
For sufficient large $X$, let $m$ be the positive integer such that $$2m^{2}  \leq X < 2(m+1)^{2},$$ that is 
\begin{equation}\label{m}
 \sqrt{\frac{X}{2}} -1   \leq X  \leq \sqrt{\frac{X}{2}}.
\end{equation}
Then from $1$ to $X$, the number of $n$ such that $h(n)\neq 0$ is at least 
\begin{equation*}
\begin{split}
&2+3+4+\dots+m+(m+1)\\
&=\frac{(m+1)(m+2)}{2}-1\\
&\geq \frac{X}{4}+\frac{1}{2\sqrt{2}}\sqrt{X} -1
\end{split}
\end{equation*}
by (\ref{m}). Hence the number of $n$ such that $h(n)$ is zero is at most 
$$\frac{3X}{4}-\frac{1}{2\sqrt{2}}\sqrt{X} +1.$$
This prove the inequality of the right hand side of (\ref{inequality1}). 

For the inequality of the left hand side of (\ref{inequality1}),
let $k$ be the maximal positive integer such that $(m+k)^{2}\leq 2(m+1)^{2}$, note that besides the numbers listed above, some other numbers $n$ still can be $\leq X$ and $h(n)\neq 0$, but they must appear in the following sequences.
\begin{equation*}
\begin{split}
&(m+1)^{2},\,\,\,\,(m+1)^{2}+1\times(m+1),\,\,\,\,(m+1)^{2}+2\times(m+1),\,\,\,\,\dots,\,\,\,(m+1)^{2}+\times(m+1)^{2},\\
&(m+2)^{2},\,\,\,\,(m+2)^{2}+1\times(m+2),\,\,\,\,(m+2)^{2}+2\times(m+2),\,\,\,\,\dots,\,\,\,(m+2)^{2}+\times(m+2)^{2},\\
&\quad\dots\,\,\,\,\,\quad\quad\quad\quad\,\,\quad\quad\,\dots\,\quad\quad\quad\quad\quad\,\quad\quad\quad\quad\dots \quad\quad\quad\quad\quad\quad\quad\quad\quad\quad\quad \dots \\
&(m+k)^{2},\,\,\,\,(m+k)^{2}+1\times(m+k),\,\,\,\,(m+k)^{2}+2\times(m+k),\,\,\,\,\dots,\,\,\,(m+k)^{2}+\times(m+k)^{2}.\\
\end{split}
\end{equation*}
Hence, the number of $n$ such that $1\leq n\leq X$ and $h(n)\neq 0$ is at most
\begin{equation*}
\begin{split}
(2+3+4+\dots+m+(m&+1))+((m+1+1)+(m+2+1)+\dots+(m+k+1))\\
&=\frac{(m+k+1)(m+k+2)}{2}-1\\
&=\frac{1}{2}(m+k)^{2}+\frac{3}{2}(m+k)\\
&\leq \frac{X}{2}+(\frac{3}{2}+\sqrt{2})\sqrt{X} +1+\frac{3\sqrt{2}}{2},
\end{split}
\end{equation*}
by the inequalities $(m+k)\leq \sqrt{2}(m+1)$ and $m+1\leq \sqrt{\frac{X}{2}}+1$. This is equivalent to the inequality of the left hand side of (\ref{inequality1}).

\noindent (2) For $n=p^{d}$ with $p$ an odd prime and $d$ an odd positive integer,  it is only to show $A_{2}(n)$ is empty. Suppose the equation $ p^{d}=x(x+y)$ has a solution $(x,y)$ with $x>1$ and $x-1 \geq y \geq 1$. We must have 
$$ x=p^{d_{1}},\,\, y=p^{d_{2}} \,\,\mbox{ with } \,\,d_{2} > d_{1}.$$
This implies that 
$$
y=p^{d_{2}}-p^{d_{1}}=p^{d_{1}}(p^{d_{2}-d_{1}}-1)> 2p^{d_{1}}=2x,
$$
since
$$
d_{2}-d_{1} \geq 1,\,\,\, p\geq 3.
$$
This contradicts the assumption $y<x-1$.

\noindent (3) for $n=p_{1}p_{2}p_{3}\dots p_{k}$ with $p_{i}$ distinct odd primes and $$p_{1}<p_{2}<p_{3}<\dots<p_{k-1}, \,\,\,p_{k}>2p_{1}p_{2}\dots p_{k-1},$$
$n$ is neither a square nor a double square, it is only to show $A_{2}(n)$ is empty. Suppose the equation $$ p_{1}p_{2}p_{3}\dots p_{k}=x(x+y)$$ has a solution $(x,y)$ with $x>1$ and $x-1 \geq y \geq 1$. We consider two cases. If $p_{k}$ is 
a prime divisor of $x$, then we have
$$
x\geq p_{k} >2p_{1}p_{2}p_{3}\dots p_{k-1} \geq 2(x+y). 
$$
Hence $y\leq -\frac{x}{2}$, this is impossible. If $p_{k}$ is a prime divisor of $x+y$, then we have
$$
x+y\geq p_{k} >2p_{1}p_{2}p_{3}\dots p_{k-1} \geq 2x,
$$
so $y>x$. This also contradicts the assumption. 

These complete the proof of Theorem \ref{theorem1}.
\end{proof}

 \begin{proof}[Proof of Theorem \ref{theorem2} ]
\noindent (1) By ( \ref{h(n)}), $h(n)=1$ if and only if $A_{n}$ is empty and $n$ is a square or double square. From
$1$ to $X$, the number of squares is about $\sqrt{X}$ and the number of double squares is  about $\sqrt{\frac{X}{2}}$. Hence the number of $1\leq n\leq X$ such that $h(n)=1$ is at most $(1+\frac{1}{\sqrt{2}})\sqrt{X}$. On the other side, for any prime
$p$, the sets $B_{p^{2}},\, C_{p^{2}}$ are all empty, hence for each $n=p^{2}$ with $1\leq p^{2}\leq X$, we have $h(n)=1$.
By The Prime Number Theorem, the number of squares of primes  is about $\frac{\sqrt{X}}{\log \sqrt{X}}=\frac{2\sqrt{X}}{\log X}$. This completes the proof of the part (1).

(2) If $n=2^{k}, \, k\geq 1$, we must have $A_{n}$ is empty. Since if we have a solution $(x,y)$ such that
$$
2^{k}=x(x+y), \,\,\, 1 \leq y \leq x-1, 
$$
then we can assume that 
$$
x=2^{k_{1}},\, x+y= 2^{k_{2}},\, k_{2}> k_{1}.
$$
Therefore, 
$$
y=2^{k_{2}}-2^{k_{1}}=2^{k_{1}}(2^{k_{2}-k_{1}}-1)\geq 2^{k_{1}}=x,
$$
which contradicts the assumption that $1\leq y \leq x-1$. Moreover, If $k$ is even, then $n$ is a square. If $k$ is odd, then
$n$ is a double square. These complete the proof of part (2) of the Theorem\ref{theorem2}.

\noindent (3) For $n=p^{2k}$ with $p$ odd prime and $k$ positive integer, $n$ can not be a double square and there is no
solution $(x,y)$ with $1\leq y\leq x-1$ such that $p^{2k}=x(x+y)$. Otherwise, we can assume that 
$$
x=p^{k_{1}},\, x+y=p^{k_{2}},\, k_{2}> k_{1}.
$$
As before
$$
y=p^{k_{2}}-p^{k_{1}}=p^{k_{1}}(p^{k_{2}-k_{1}}-1)\geq 2p^{k_{1}}=2x,
$$
which is not possible.

\noindent (4) For $n=p_{1}^{2}p_{2}^{2}$ with the assumption that $p_{1}$ and $p_{2}$ are odd primes with $p_{2} > \sqrt{2}p_{1}$. it is only to show that the equation 
$$
2^{k}=x(x+y), \,\,\, 1 \leq y \leq x-1, 
$$
has no solutions. We note that $x$ can be $p_{1}$ or $p_{1}^{2}$. for the case $x=p_{1}$, we have 
$$
x+y=p_{1}p_{2}^{2},\,\mbox{  hence } y=p_{1}(p_{2}^{2}-1) > 2p_{1}.
$$
For the case $x=p_{1}^{2}$, we have $x+y=p_{2}^{2}$, hence $y=p_{2}^{2}-p_{1}^{2}>p_{1}^{2}=x$, since $p_{2} > \sqrt{2}p_{1}$. This shows the equation above has no solutions.
The proof of Theorem\ref{theorem2} is completed. 
\end{proof}

\begin{proof}[Proof of Theorem \ref{theorem3}]
(1) We note that
\begin{align*}
 \# \{ 1\leq n \leq X| h(n) =2   \} < \# \{ 1\leq n \leq X| h(n) \neq 0   \}. 
\end{align*}
The right hand side of the inequality above is at most $$\frac{X}{2}+(\frac{3}{2}+\sqrt{2})\sqrt{X} +1+\frac{3\sqrt{2}}{2},$$ which is proved in part one of theorem\ref{theorem1}.  As for the left hand side of (\ref{1-theorem3}), this follows from the part (3) of the  current theorem. Since from $1$ to $X$, there are about $ \frac{\sqrt{x}}{\log \sqrt{X}}$ numbers which are the forms $p(p+1)$ by the Prime Number Theorem.

(2) For $n=p_{1}p_{2}$ with $p_{1} <p_{2}<2p_{1}$. Firstly, by a theorem of Betrand-Chebyshev, for any prime $p_{1}$, there is at least one $p_{2}$ satisfying $p_{1} <p_{2}<2p_{1}$. Moreover,
 $n$ can not be a square or a double square and the equation
$$
p_{1}p_{2}=x(x+y), \,\,\, 1 \leq y \leq x-1
$$
has a unique solution $(x, y)=(p_{1}, p_{2})$ by the condition $p_{1} <p_{2}<2p_{1}$. Hence $h(n)=2$.

(3) For $n=p(p+1)$ with $p$ odd prime. since $p$ and $p+1$ are coprime, hence $n$ can not be a square or a double
square. Consider the equation
$$
p(p+1)=x(x+y), \,\,\, 1 \leq y \leq x-1,
$$
the only solution is $x=p$ and $y=1$, therefore $h(n)=2$.  The analysis for the case $n=(p-1)p$ is similar.
\end{proof}

\begin{proof}[Proof of Theorem \ref{theorem4}]
(1) For $n=p_{1}^{2}p_{2}^{2}$ with $p_{1}$ and $p_{2}$ odd primes and $p_{1} < p_{2} < \sqrt{2}p_{1}$, by a theorem  in
\cite{KL}, for any sufficient large $p_{1}$, there is at least one prime $p_{2}$ satisfying $p_{1} < p_{2} < \sqrt{2}p_{1}$. $n$
can not be a double square, it is only to show there is only one pair $(x,y)$ with $1 \leq y \leq x-1$ satisfying the equation
$$
p_{1}^{2}p_{2}^{2}=x(x+y).
$$
Here, $x$ can be $p_{1}$, $p_{1}^{2}$, $p_{2}$.  $x=p_{1}$ is impossible, this results in $y=p_{1}(p_{2}^{2}-1)>2p_{1}=2x$.
$x=p_{2}$ is also impossible, it also results  in  $y=p_{2}(p_{1}^{2}-1)>2p_{2}=2x$.  The only solution to the equation above is $(x, y)=(p_{1}^{2}, p_{2}^{2})$, because of the condition $p_{1} < p_{2} < \sqrt{2}p_{1}$.

\end{proof}
\begin{proof}[Proof of Theorem \ref{theorem5}]
 By (\ref{rep of h(n)}), 
$$
h(q) \equiv \sum_{n=1}^{\infty} (q^{n^{2}}+q^{2n^{2}}) \pmod{2},
$$
$h(n)$ is odd if and only if $n$ is a square or a double square and the destiny of square numbers and double square numbers is zero.  Therefore, for almost all integers $n$, $h(n)$ is even.
\end{proof}

\section{Questions for further study}
We use  some quadratic equations to study the small values of $h(n)$, but it is unclear whether there is a relation
between the function $h(q)$ and quadratic fields. This is worth investigating. For the distributions of the coefficients of $h(q)$, the experiments suggest more should be true. Here we list our conjectures based on our experiments.
\begin{conjecture}
$h(n)$ can attain each positive integers infinite many times, In particular, we have
$
\limsup h(n) = \infty.
$
\end{conjecture}

 This conjectural property  for  $h(n)$ is similar to that of $S(n)$, which is our original motivating of investigating the function $h(q)$.  However, the following conjecture shows $h(q)$ is completely different from $\sigma(q)$.

\begin{conjecture}

For sufficient large $X$,
\begin{align*}
\# \{ 1\leq n \leq X| h(n) =0  \} \asymp  \frac{3X}{4}-\frac{1}{2\sqrt{2}}\sqrt{X},
\end{align*} 
 \end{conjecture}

 \begin{conjecture}
For sufficient large $X$,
\begin{align*}
\# \{ 1\leq n \leq X| h(n) =2   \} \asymp \frac{X}{4}+\frac{1}{2\sqrt{2}}\sqrt{X}.
\end{align*}
This implies that almost all nonzero values are $2$.
\end{conjecture}
Our Theorem\ref{theorem5} coincides with the  Conjecture 1 and Conjecture 2.

\begin{conjecture}
For sufficient large $X$,
\begin{align*}
\# \{ 1\leq n \leq X| h(n) =1  \} \asymp\left(1+ \frac{1}{\sqrt{2}}\right)\sqrt{X}.
\end{align*}
\end{conjecture}

More difficult questions are what are the asymptotical behaviors of
\begin{align*}
\# \{ 1\leq n \leq X| h(n) =2k   \}
\end{align*}
and
\begin{align*}
\# \{ 1\leq n \leq X| h(n) =2k +1  \}
\end{align*}
for a fixed positive interer $k$ and sufficient large $X$?

\begin{center}
\begin{table}
\caption{The number of $1\leq n\leq X$ such that $h(n)$ takes values from $0$ to $16$ for $X\leq 100000$.}
\begin{tabular}{|c|cc  ccc c|cc ccc}
\hline
  $h(n)$     &  $X=1000$ &  $X= 5000$ &$X= 10000$ & $X= 20000$& $X= 50000$ & $X= 100000$   \\
\hline
   0    &   664 &  3486& 7068    &   14312 & 36249 &   73130\\
   1   &   44&   93 &   129    &   179 & 275 &   380\\
2    &   255 &   1181 &  2300    &   4455 & 10718 &   20798\\
   3   &   9 &   23 &   32    &  46 & 74 &   108\\            
    4    &  27 &  180 &   369    &   758 & 1944 &   3969\\           
   5   &   0 &  3 &   7    &   11 & 21 &   31\\            
    6    &   1 &   32 &   80    &   186 & 509 &   1051\\           
    7    &  0 &   1 &   2    &   5 & 8 &   14\\            
   8    &  0 &   1 &   13    &  41 & 148 &   354\\           
   9    &   0 &   0 &  0    &   0 & 3 &   5\\            
   10    &   0 &   0 &   0    &   7 & 43 &  120\\            
     11    &   0 &   0 &   0    &   0 & 0 &   1\\           
   12   &   0 &   0 &   0    &   0 & 8 &   28\\           
 13    &   0 &   0 &   0    &   0 & 0 &   0\\
   14   &  0 &   0 &   0    &   0 & 0 &   10\\
 15    &   0 &   0 &   0    &   0 & 0 &   0\\    
 16    &  0 &  0 &   0    &  0 & 0 &   1\\                       
 \hline
\end{tabular}
\end{table}
\end{center}

\section*{Acknowledgements}
The author would like to thank Professor Peter Paule for his valuable comments on an earlier version of this paper and encouragements. 


\end{document}